\documentclass{amsart}

\usepackage{setspace}
\usepackage{graphicx}
\usepackage{amscd,amssymb,amsthm}

\usepackage{hyperref}
\usepackage{graphicx}
\usepackage{amsmath, amsthm, amscd, amsfonts, amssymb, graphicx, color}
 \usepackage{url}
\usepackage{enumerate}
 \makeatletter
\let\reftagform@=\tagform@
\def\tagform@#1{\maketag@@@{(\ignorespaces\textcolor{blue}{#1}\unskip\@@italiccorr)}}
\renewcommand{\eqref}[1]{\textup{\reftagform@{\ref{#1}}}}
\makeatother

\hypersetup{colorlinks=true, linkcolor=red, anchorcolor=green,
citecolor=cyan, urlcolor=red, filecolor=magenta, pdftoolbar=true}

\usepackage{epsfig}        
\usepackage{epic,eepic}       

\newtheorem{theorem}{Theorem}
\theoremstyle{plain}

\newtheorem{corollary}{Corollary}

\newtheorem{remark}{Remark}

\numberwithin{equation}{section}

\begin{document}

\title[ An Application of Hayashi's Inequality]{An Application of
Hayashi's Inequality for Differentiable Functions}
\author[ Alomari \& Bakula]{Mohammad W. Alomari$^1$ and Milica K. Bakula$^{2,*}$}

\thanks{$^{2,*}$Corresponding author}

\address{$^1$Department of Mathematics, Faculty of Science and
	Information Technology, Irbid National University, P.O. Box 2600,
	Irbid, P.C. 21110, Jordan.} \email{$^1$mwomath@gmail.com}

\address{$^{2,*}$Faculty of Science, University of Split, Split, Croatia}
\email{$^{2,*}$milica@pmfst.hr}

\keywords{Hayashi's Inequality, Ostrowski's inequality,
Differentiable functions} \subjclass[2010]{26D15}

\begin{abstract}
In this work, we offer new applications of Hayashi's inequality for differentiable functions by proving new error estimates of the Ostrowski and trapezoid type quadrature rules.
\end{abstract}

\maketitle

\section{Introduction\label{sec1}}
 
The Hayashi's inequality  states that (\cite{MPF}, pp. 311-312):
\begin{theorem}
Let $p:[a,b] \to \mathbb{R}$  be a nonincreasing mapping on $[a,
b]$ and $h : [a, b] \to \mathbb{R}$ an integrable mapping on $[a,
b]$ with $0 \le h(x) \le A$ for all $x\in [a,b]$. Then, the
inequality
\begin{align}
A \int_{b-\lambda}^b{p(x)dx} \le  \int_{a}^b{p(x)h(x)dx} \le A
\int_a^{a+\lambda}{p(x)dx} \label{eq1.1}
\end{align}
holds, where $\lambda = \frac{1}{A}\int_{a}^b{h(x)dx}$.
\end{theorem}
This inequality is a simple generalization of Steffensen's
inequality which holds with same assumptions with $A=1$. For recent results concerning the Hayashi's inequality see \cite{alomari4}.

In 1996, Agarwal and Dragomir \cite{Agarwal} presented an
application of this inequality, as follows:
\begin{theorem}
Let $f : I \subseteq \mathbb{R} \to \mathbb{R} $  be a
differentiable mapping on $I^{\circ}$ (the interior of $I$) and
$[a, b] \subseteq I^{\circ}$  with $M = \mathop {\sup }\limits_{x
\in \left[ {a,b} \right]} f^{\prime}\left( x \right) < \infty$, $m
= \mathop {\inf }\limits_{x \in \left[ {a,b} \right]}
f^{\prime}\left( x \right) < \infty$ and $M > m$. If $f^{\prime}$
is integrable on $[a, b]$, then the following inequality holds
\begin{align}
&\left|{\frac{1}{b-a}\int_a^b{f\left(t\right)dt}-\frac{
f\left(a\right)+ f\left(b\right)}{2} }\right|\label{eq1.2}
\\
&\le \frac{{\left[ {f\left( b \right) - f\left( a \right) -
m\left( {b - a} \right)} \right]\left[ {M\left( {b - a} \right) -
f\left( b \right) + f\left( a \right)} \right]}}{{2\left( {M - m}
\right)\left( {m - a} \right)}}
\nonumber\\
&\le \frac{\left( {M - m} \right)\left( {b - a}
\right)}{8}.\nonumber
\end{align}
\end{theorem}
This elegant inequality presents an error estimation for the Trapezoid rule.
 
In 2002, Gauchman \cite{Gauchman} generalized \eqref{eq1.2} for
$n$-times differentiable using Taylor expansion. So that
\eqref{eq1.2} become a special case of Gauchman result when $n=0$.

In this work, a generalization of  \eqref{eq1.2} is obtained.   In the same argument, other two inequalities of Ostrowski's and Trapezoid type  are also introduced.

\section{The Results\label{sec2}}
Let us begin with the following generalization of \eqref{eq1.2}.
\begin{theorem}\label{thm3}
Let $g:[a,b] \to \mathbb{R}$ be an absolutely continuous function
on $[a,b]$ with $0\le g^{\prime}\left(t\right) \le
\left(b-a\right)$ and $(\cdot - t) g^{\prime}$ is integrable on
$[a,b]$. Then
\begin{multline}
\left|{\frac{1}{b-a}\int_a^b{g\left(t\right)dt}-
\frac{\left(x-a\right)g\left(a\right)+\left(b-x\right)g\left(b\right)}{b-a}-
\lambda \left({x-\frac{a+b }{2} }\right)}\right|
\\
\le
 \frac{\lambda }{2} \left({b-a-\lambda}\right)\le
 \frac{\left({b-a}\right)^2}{8}  \label{eq2.1}
\end{multline}
for all $x\in [a,b]$. The equality satisfied when
$g\left(x\right)=x$, $x\in [0,1]$. In particular, for
$x=\frac{a+b}{2}$
\begin{align*}
\left|{\frac{1}{b-a}\int_a^b{g\left(t\right)dt} -\frac{
g\left(a\right)+ g\left(b\right)}{2}}\right| \le \frac{\lambda
}{2} \left({b-a-\lambda}\right)\le
 \frac{\left({b-a}\right)^2}{8},
\end{align*}
where  $\lambda=\frac{g\left(b\right)-g\left(a\right)}{b-a}$.
\end{theorem}

\begin{proof}
Let $f\left(t\right)=x-t$, $t\in [a,b]$. Applying the Hayashi's
inequality \eqref{eq1.1} by setting
$p\left(t\right)=f\left(t\right)$ and
$h\left(t\right)=g^{\prime}\left(t\right)$, we get
\begin{align}
&\left(b-a\right) \int_{b-\lambda}^b{(x-t)dt} \le
\int_a^b{(x-t)g^{\prime}\left(t\right)dt} \le \left(b-a\right)
\int_a^{a+\lambda}{(x-t)dt}\label{eq2.2}
\end{align}
where, $A= b-a$  or we write
\begin{align*}
\lambda = \frac{1}{b-a} \int_a^b{g^{\prime}\left(t\right)dt} =
\frac{g\left(b\right)-g\left(a\right)}{b-a}.
\end{align*}
Also, we have
\begin{align*}
 \int_{b-\lambda}^b{(x-t)dt}
=\lambda\left(x-b\right)+\frac{1}{2} \lambda^2,
\end{align*}
\begin{align*}
\int_a^b{(x-t)g^{\prime}(t)dt}=-\left(x-a\right)g\left(a\right)-\left(b-x\right)g\left(b\right)+\int_a^b{g\left(t\right)dt},
\end{align*}
and
\begin{align*}
\int_a^{a+\lambda}{(x-t)dt}=  \lambda \left(x-a\right)
-\frac{1}{2}\lambda^2.
\end{align*}
Substituting the above equalities in \eqref{eq2.2} and dividing by
$(b-a)$, we get
\begin{align*}
&\ell_1(x):=\lambda\left(x-b\right) +\frac{1}{2} \lambda^2
\\
&\le
-\frac{\left(x-a\right)g\left(a\right)+\left(b-x\right)g\left(b\right)}{b-a}+\frac{1}{b-a}\int_a^b{g\left(t\right)dt}
:=I
\\
&\le\lambda \left(x-a\right) -\frac{1}{2}   \lambda^2 :=\ell_2(x).
\end{align*}
We also have
\begin{align*}
\left|{I-\frac{\ell_1(x)+\ell_2(x)}{2} }\right| =\left|{I- \lambda
\left({x-\frac{a+b }{2} }\right) }\right| &\le
\frac{\ell_2(x)-\ell_1(x)}{2}
\\
&= \frac{\lambda }{2} \left({b-a-\lambda}\right)
\end{align*}
which proves the first inequality in \eqref{eq2.1}. The second
inequality follows by applying the same technique in
(\cite{Agarwal}, pp. 96-97).
\end{proof}

\begin{remark}
For very close related results of Theorem \ref{thm3}, we refer the
reader to \cite{alomari1}--\cite{Guessab} and \cite{Ujevic}.
\end{remark}

The corrected generalized version of Agarwal-Dragomir result
\eqref{eq1.2} is incorporated in the following corollary.
\begin{corollary}\label{cor1}
Let $g:[a,b] \to \mathbb{R}$ be an absolutely continuous function
on $[a,b]$ with $\gamma\le g^{\prime}\left(t\right) \le \Gamma$
and $(\cdot - t) g^{\prime}$ is integrable on $[a,b]$. Then
\begin{multline}
\left|{\frac{1}{b-a}\int_a^b{g\left(t\right)dt}
-\frac{\left(x-a\right)g\left(a\right)+\left(b-x\right)g\left(b\right)}{b-a}
-\frac{g\left( {b} \right)-g\left( {a} \right)}{b-a}\cdot \left(
{x - \frac{{a + b}}{2}} \right)}\right|
\\
\le
 \frac{\Gamma -\gamma}{2} \cdot \lambda\cdot\frac{\left({b-a-\lambda}\right)}{b-a}
 \le  \frac{\left({\Gamma -\gamma}\right)\left({b-a}\right)}{8},\label{eq2.3}
\end{multline}
for all $x\in [a,b]$. In particular, for $x=\frac{a+b}{2}$
\begin{align*}
\left|{\frac{1}{b-a}\int_a^b{g\left(t\right)dt} -\frac{
g\left(a\right)+ g\left(b\right)}{2}}\right| \le \frac{\Gamma
-\gamma}{2} \cdot
\lambda\cdot\frac{\left({b-a-\lambda}\right)}{b-a}\le
\frac{\left({\Gamma -\gamma}\right)\left({b-a}\right)}{8},
\end{align*}
where $\lambda  = \frac{{g\left( b \right) - g\left( a \right) -
\gamma \left({b-a}\right)}}{{\Gamma  - \gamma }}$.
\end{corollary}
\begin{proof}
Repeating the proof of Theorem \ref{thm3}, with
$h(t)=g^{\prime}(t)-\gamma$, $t\in [a,b]$, we get  the first
inequality. The second inequality in \eqref{eq2.3} follows by
applying the same technique in (\cite{Agarwal}, pp. 96-97). Some
manipulations in  the particular $x=\frac{a+b}{2}$ gives the same
result in \eqref{eq1.2}.
\end{proof}

\begin{remark}
Let the assumptions of Corollary \ref{cor1} be satisfied. Then,
the Corollaries 3-4 and the Remarks 1-2 in \cite{Agarwal}  (p. 97)
are hold.
\end{remark}

In 1997, Dragomir and Wang \cite{Dragomir} introduced an
inequality of Ostrowski--Gr\"{u}ss' type as follows:
\begin{align}
\label{eq2.4}\left| {f\left( x \right) - \frac{1}{{b - a}}\int_a^b
{f\left( t \right)dt}  - \frac{{f\left( b \right) - f\left( a
\right)}}{{b - a}}\left( {x - \frac{{a + b}}{2}} \right)} \right|
\le \frac{1}{4}\left( {b - a} \right)\left( {\Gamma  - \gamma }
\right),
\end{align}
holds for all $x \in [a,b]$, where  $f$ is assumed to be
 differentiable
 $[a,b$ with  $f^{\prime }\in L^1[a,b]$ and $\gamma
\le f'\left( x \right) \le \Gamma$, $\forall x \in [a,b]$.

In 1998, another result for twice differentiable was proved in
\cite{Cerone1}. In 2000, the constant $\frac{1}{4}$ in
\eqref{eq2.4} was improved  by $\frac{1}{\sqrt{3}}$ in
\cite{Matic}.

A better  improvement of \eqref{eq2.4} can be deduced by applying
the Hayashi's inequality as presented in the following result.
\begin{theorem}\label{thm4}
Let $g:[a,b] \to \mathbb{R}$ be an absolutely continuous function
on $[a,b]$ with $0\le g^{\prime}\left(t\right) \le
\left(b-a\right)$  and $\left(\cdot-t\right)
g^{\prime}\left(t\right)$ is integrable on $[a,b]$. Then
\begin{align}
\left|{\frac{1}{b-a}\int_a^b{g\left(t\right)dt}- g\left(x\right) +
\lambda \left( {x - \frac{{a + b}}{2}} \right) }\right|\le \lambda
\frac{b-a}{2} - \lambda^2 \le
\frac{\left(b-a\right)^2}{16},\label{eq2.5}
\end{align}
for all $x\in [a,b]$, where $\lambda =
\frac{g\left(b\right)-g\left(a\right)}{b-a}$. In particular, for
$x=\frac{a+b}{2}$
\begin{align*}
\left|{\frac{1}{b-a}\int_a^b{g\left(t\right)dt}- g\left( \frac{{a
+ b}}{2}\right)}\right| \le \lambda \frac{b-a}{2} - \lambda^2 \le
\frac{\left(b-a\right)^2}{16}.
\end{align*}
\end{theorem}

\begin{proof}
Fix  $x\in [a,b]$ and let $f\left(t\right)=a-t$, $t\in [a,x]$.
Applying the Hayashi's inequality \eqref{eq1.1} by setting
$p\left(t\right)=f\left(t\right)$ and
$h\left(t\right)=g^{\prime}\left(t\right)$, we get
\begin{align}
&\left(b-a\right) \int_{x-\lambda}^x{(a-t)dt} \le
\int_a^x{(a-t)g^{\prime}\left(t\right)dt} \le \left(b-a\right)
\int_a^{a+\lambda}{(a-t)dt}\label{eq2.6}
\end{align}
where,
\begin{align*}
\lambda = \frac{1}{b-a} \int_a^b{g^{\prime}\left(t\right)dt} =
\frac{g\left(b\right)-g\left(a\right)}{b-a}.
\end{align*}
Also, we have
\begin{align*}
 \int_{x-\lambda}^x{(a-t)dt}
=-\lambda\left(x-a\right) + \frac{1}{2} \lambda^2,
\end{align*}
\begin{align*}
\int_a^x{(a-t)g^{\prime}(t)dt}=-\left(x-a\right)g\left(x\right)
+\int_a^x{g\left(t\right)dt},
\end{align*}
and
\begin{align*}
\int_a^{a+\lambda}{(a-t)dt}= -\frac{1}{2}\lambda^2.
\end{align*}
Substituting in \eqref{eq2.6}, we have
\begin{align}
\left( {b -a} \right) \left({-\lambda\left(x-a\right) +
\frac{1}{2} \lambda^2}\right) \le -\left(x-a\right)g\left(x\right)
+\int_a^x{g\left(t\right)dt}\le -\frac{1}{2}\lambda^2\left( {b -a}
\right).\label{eq2.7}
\end{align}

Now, let $f\left(t\right)=b-t$, $t\in [x,b]$. Applying Hayashi's
inequality \eqref{eq1.1} again we get
\begin{align}
&\left(b-a\right) \int_{b-\lambda}^b{(b-t)dt} \le
\int_x^b{(b-t)g^{\prime}\left(t\right)dt} \le \left(b-a\right)
\int_x^{x+\lambda}{(b-t)dt}\label{eq2.8}
\end{align}
where,
\begin{align*}
 \int_{b-\lambda}^b{(b-t)dt}
=\frac{1}{2} \lambda^2,
\end{align*}
\begin{align*}
\int_x^b{(b-t)g^{\prime}(t)dt}= -\left(b-x\right)g\left(x\right)
+\int_x^b{g\left(t\right)dt},
\end{align*}
and
\begin{align*}
\int_x^{x+\lambda}{(b-t)dt}= \lambda
\left(b-x\right)-\frac{1}{2}\lambda^2.
\end{align*}
Substituting in \eqref{eq2.8}, we have
\begin{align}
\frac{1}{2} \lambda^2 \left( {b -a} \right)\le
-\left(b-x\right)g\left(x\right)
+\int_x^b{g\left(t\right)dt}\le\left( {b -a} \right)
\left({\lambda
\left(b-x\right)-\frac{1}{2}\lambda^2}\right).\label{eq2.9}
\end{align}
Adding \eqref{eq2.7} and \eqref{eq2.9} we get
\begin{align*}
\left( {b -a} \right)\left({-\lambda\left(x-a\right)
+\lambda^2}\right)&\le
\int_a^b{g\left(t\right)dt}-\left(b-a\right)g\left(x\right)
\nonumber\\
&\le\left( {b -a} \right) \left({\lambda \left(b-x\right)-
\lambda^2}\right).
\end{align*}
Setting $$I:=\frac{1}{b-a} \int_a^b{g\left(t\right)dt}-
g\left(x\right),$$
$$\ell_1(x)= \left({-\lambda\left(x-a\right) +\lambda^2}\right),$$ and
$$\ell_2(x)=  \left({\lambda \left(b-x\right)- \lambda^2}\right).$$
Therefore,
\begin{align*}
\left|{I-\frac{\ell_1(x)+\ell_2(x)}{2}
}\right|&=\left|{\frac{1}{b-a} \int_a^b{g\left(t\right)dt}-
g\left(x\right) + \lambda \left( {x - \frac{{a + b}}{2}} \right)
}\right|
\\
&\le \frac{\ell_2(x)-\ell_1(x)}{2}
\\
&= \lambda \frac{b-a}{2} - \lambda^2,
\end{align*}
which proves the first inequality in \eqref{eq2.5}. To prove the
second inequality, define the mapping
$\phi(t)=-t^2+\frac{b-a}{2}t$, then $\max \phi(t) =
\phi\left(\frac{b-a}{4}\right)=\left(\frac{b-a}{4}\right)^2$, so
that $\phi(\lambda)=-\lambda^2+\frac{b-a}{2}\lambda \le
\left(\frac{b-a}{4}\right)^2$, which completes the proof of the
theorem.
\end{proof}

\begin{corollary}\label{cor2}
Let $g:[a,b] \to \mathbb{R}$ be an absolutely continuous function
on $[a,b]$ with $\gamma\le g^{\prime}\left(t\right) \le\Gamma$ and
$\left(\cdot-t\right) g^{\prime}\left(t\right)$ is integrable on
$[a,b]$. Then
\begin{align}
&\left|{\frac{1}{b-a} \int_a^b{g\left(t\right)dt}- g\left(x\right)
 +\frac{g\left( {b} \right) -g\left( {a} \right)}{b-a}\cdot \left( {x -
\frac{{a + b}}{2}} \right) }\right|\label{eq2.10}
\\
&\le \left( {\frac{\Gamma-\gamma}{b-a}} \right) \left(\lambda
\frac{b-a}{2} - \lambda^2\right) \le \frac{\left(b-a\right)\left(
{\Gamma-\gamma} \right)}{16},\nonumber
\end{align}
for all $x\in [a,b]$, where $\lambda  = \frac{{g\left( b \right) -
g\left( a \right) - \gamma \left({b-a}\right)}}{{\Gamma  - \gamma
}}$. In particular, for $x=\frac{a+b}{2}$
\begin{align*}
\left|{\frac{1}{b-a}\int_a^b{g\left(t\right)dt}- g\left( \frac{{a
+ b}}{2}\right)}\right| &\le  \left( {\frac{\Gamma-\gamma}{b-a}}
\right) \left(\lambda \frac{b-a}{2} - \lambda^2\right)
\\
&\le \frac{\left(b-a\right)\left( {\Gamma-\gamma} \right)}{16}.
\end{align*}
\end{corollary}

\begin{proof}
Repeating the proof of Theorem \ref{thm4}, with
$h(t)=g^{\prime}(t)-\gamma$, $t\in [a,b]$, we get  the first
inequality. The second inequality \eqref{eq2.10} follows by
applying the same technique in the proof.
\end{proof}

\begin{remark}
As we notice, \eqref{eq2.10} improves \eqref{eq2.4} by
$\frac{1}{4}$, which is better than Mati\'{c} et al.  improvement
in \cite{Matic}.
\end{remark}


In \cite{alomari2}, under the assumptions of Theorem \ref{thm4},
the author of this paper proved the following version of
Guessab--Schmeisser type inequality (see \cite{Guessab}):
\begin{align}
&\left|{\frac{{g\left( x \right) + g\left( {a + b - x}
\right)}}{2} -\frac{1}{b-a}\int_{a}^b{g\left(t\right)dt}}\right|
\le
\frac{\left(\Gamma-\gamma\right)\left(b-a\right)}{8},\label{eq2.11}
\end{align}
for all $x\in \left[a,\frac{a+b}{2}\right]$.

An improvement of \eqref{eq2.11}  is considered as follows:
\begin{theorem}\label{thm5}
Let $g:[a,b] \to \mathbb{R}$ be an absolutely continuous function
on $[a,b]$ with $0\le g^{\prime}\left(t\right) \le
\left(b-a\right)$  and $\left(\cdot-t\right)
g^{\prime}\left(t\right)$ is integrable on $[a,b]$. Then
\begin{align}
\left|{\frac{1}{b-a}\int_{a}^b{g\left(t\right)dt}- \frac{{g\left(
x \right) + g\left( {a + b - x} \right)}}{2} }\right| \le \lambda
\left[ {\frac{{b - a}}{2} - \frac{3}{2}\lambda } \right]\le
\frac{\left(b-a\right)^2}{24},\label{eq2.12}
\end{align}
for all $x\in \left[a,\frac{a+b}{2}\right]$, where $\lambda =
\frac{g\left(b\right)-g\left(a\right)}{b-a}$.
\end{theorem}

\begin{proof}
Fix $x\in \left[a,\frac{a+b}{2}\right]$ and let
$f\left(t\right)=a-t$, $t\in [a,x]$. Applying Hayashi's inequality
\eqref{eq1.1} by setting $p\left(t\right)=f\left(t\right)$ and
$h\left(t\right)=g^{\prime}\left(t\right)$, then \eqref{eq2.6}
holds, i.e.,
\begin{align}
\left( {b -a} \right) \left({-\lambda\left(x-a\right) +
\frac{1}{2} \lambda^2}\right) &\le
-\left(x-a\right)g\left(x\right)
+\int_a^x{g\left(t\right)dt}\label{eq2.13}
\\
&\le -\frac{1}{2}\lambda^2\left( {b -a} \right).\nonumber
\end{align}
where,
\begin{align*}
\lambda = \frac{1}{b-a} \int_a^b{g^{\prime}\left(t\right)dt} =
\frac{g\left(b\right)-g\left(a\right)}{b-a}.
\end{align*}
Now, let $f\left(t\right)=\frac{a+b}{2}-t$, $t\in
\left[x,a+b-x\right]$. Applying the Hayashi's inequality
\eqref{eq1.1} again we get
\begin{align}
\left(b-a\right)
\int_{a+b-x-\lambda}^{a+b-x}{\left(\frac{a+b}{2}-t\right)dt}
\le\int_x^{a+b-x}{\left(\frac{a+b}{2}-t\right)g^{\prime}\left(t\right)dt}\label{eq2.14}
\\
\le \left(b-a\right)
\int_x^{x+\lambda}{\left(\frac{a+b}{2}-t\right)dt}\nonumber
\end{align}
where,
\begin{align*}
 \int_{a+b-x-\lambda}^{a+b-x}{\left(\frac{a+b}{2}-t\right)dt}
= -\lambda \left( {\frac{{a + b}}{2} - x} \right) +
\frac{1}{2}\lambda ^2,
\end{align*}
\begin{multline*}
\int_{x}^{a+b-x}{\left(\frac{a+b}{2}-t\right)g^{\prime}(t)dt}
\\
= - \left( {\frac{{a + b}}{2} - x} \right)\left( {g\left( x
\right) + g\left( {a + b - x} \right)}
\right)+\int_x^{a+b-x}{g\left(t\right)dt},
\end{multline*}
and
\begin{align*}
\int_x^{x+\lambda}{\left(\frac{a+b}{2}-t\right)dt}= \lambda \left(
{\frac{{a + b}}{2} - x} \right) - \frac{1}{2}\lambda ^2.
\end{align*}
Substituting in \eqref{eq2.14}, we have
\begin{align}
&-\left( {b -a} \right)\left[{\lambda \left( {\frac{{a + b}}{2} -
x} \right) - \frac{1}{2}\lambda ^2}\right]
\nonumber\\
&\le  - \left( {\frac{{a + b}}{2} - x} \right)\left( {g\left( x
\right) + g\left( {a + b - x} \right)}
\right)+\int_x^{a+b-x}{g\left(t\right)dt}\label{eq2.15}
\\
&\le \left( {b -a} \right)\left[{\lambda \left( {\frac{{a + b}}{2}
- x} \right) - \frac{1}{2}\lambda ^2}\right].\nonumber
\end{align}
Now, let $f\left(t\right)=b-t$, $t\in [a+b-x,b]$. Applying
Hayashi's inequality \eqref{eq1.1} again we get
\begin{align}
\left(b-a\right) \int_{b-\lambda}^b{\left(b-t\right)dt} &\le
\int_{a+b-x}^b{\left(b-t\right)g^{\prime}\left(t\right)dt}
\label{eq2.16}\\
&\le \left(b-a\right)
\int_{a+b-x}^{a+b-x+\lambda}{\left(b-t\right)dt}\nonumber
\end{align}
where,
\begin{align*}
 \int_{b-\lambda}^b{(b-t)dt}
=\frac{1}{2} \lambda^2,
\end{align*}
\begin{align*}
\int_{a+b-x}^b{\left(b-t\right)g^{\prime}(t)dt}=
-\left(x-a\right)g\left({a+b-x}\right)
+\int_{a+b-x}^b{g\left(t\right)dt},
\end{align*}
and
\begin{align*}
\int_{a+b-x}^{a+b-x+\lambda}{(b-t)dt}= \lambda
\left(x-a\right)-\frac{1}{2}\lambda^2.
\end{align*}
Substituting in \eqref{eq2.16}, we have
\begin{align}
\frac{1}{2} \lambda^2 \left( {b -a} \right)&\le
-\left(x-a\right)g\left({a+b-x}\right)
+\int_{a+b-x}^b{g\left(t\right)dt}\label{eq2.17}
\\
&\le\left( {b -a} \right)  \left[{\lambda
\left(x-a\right)-\frac{1}{2}\lambda^2 }\right].\nonumber
\end{align}
Adding \eqref{eq2.13}, \eqref{eq2.15} and \eqref{eq2.17} we get
\begin{align*}
- \lambda \left( {b - a} \right)\left[ {\frac{{b - a}}{2} -
\frac{3}{2}\lambda } \right] &\le \int_{a}^b{g\left(t\right)dt}-
\left( {b - a} \right)\frac{{g\left( x \right) + g\left( {a + b -
x} \right)}}{2}
\\
&\le  \lambda \left( {b - a} \right)\left[ {\frac{{b - a}}{2} -
\frac{3}{2}\lambda } \right],
\end{align*}
which is equivalent to the first inequality in \eqref{eq2.12}. To
prove the second inequality in \eqref{eq2.12}, define the mapping
$\phi(t)=-\frac{3}{2}\left( {b - a} \right) t^2+\frac{\left( {b -
a} \right)^2}{2}t$, then $\max \phi(t) =
\phi\left(\frac{b-a}{6}\right)=\frac{\left(b-a\right)^2}{24}$, so
that $\phi\left(\lambda\right)=-\frac{3}{2}\left( {b - a} \right)
\lambda^2+\frac{\left( {b - a} \right)^2}{2}\lambda \le
\frac{\left(b-a\right)^2}{24}$,  which completes the proof of the
theorem.
\end{proof}
A generalization of \eqref{eq2.11} and \eqref{eq2.12} is
incorporated in the following result.
\begin{corollary}\label{cor3}
Let $g:[a,b] \to \mathbb{R}$ be an absolutely continuous function
on $[a,b]$ with $\gamma\le g^{\prime}\left(t\right) \le\Gamma$ and
$\left(\cdot-t\right) g^{\prime}\left(t\right)$ is integrable on
$[a,b]$. Then
\begin{align}
&\left|{\frac{1}{b-a}\int_{a}^b{g\left(t\right)dt}- \frac{{g\left(
x \right) + g\left( {a + b - x} \right)}}{2} }\right|
\nonumber\\
&\le \frac{1}{2}\lambda  \left( {\frac{{\Gamma  - \gamma }}{{b -
a}}} \right)  \left[ {\left( {b - a} \right) - 3\lambda }
\right]\le \frac{\left( {\Gamma - \gamma}
\right)\left(b-a\right)}{24}, \label{eq2.18}
\end{align}
for all $x\in \left[a,\frac{a+b}{2}\right]$, where $\lambda  =
\frac{{g\left( b \right) - g\left( a \right) - \gamma
\left({b-a}\right)}}{{\Gamma  - \gamma }}$.
\end{corollary}

\begin{proof}
Repeating the proof of Theorem \ref{thm5}, with
$h(t)=g^{\prime}(t)-\gamma$, $t\in [a,b]$, we get the first
inequality. The second inequality in \eqref{eq2.18} follows by
applying the same technique in the proof.
\end{proof}

\section{Applications for P.D.F.'S}

Let X be a random variable taking values in the finite interval
$[a,b]$, with the probability density function $f : [a,b] \to
[0,1]$ with the cumulative distribution function $F(x) = Pr(X \le
x) = \int_a^b{f(t)dt}$.

\begin{theorem}
	\label{thm6}With the assumptions of Theorem \ref{thm4}, we have
	the inequality
	\begin{align*}
		\left| {\frac{1}{2}\left[ {F\left( x \right) + F\left( {a + b - x}
				\right)} \right] - \frac{{b - E\left( X \right)}}{{b - a}}}
		\right|
		\le \frac{1}{2}\lambda  \left( {\frac{{\Gamma  - \gamma }}{{b -
					a}}} \right)  \left[ {\left( {b - a} \right) - 3\lambda }
		\right]\le \frac{\left( {\Gamma - \gamma}
			\right)\left(b-a\right)}{24}  
	\end{align*}
for all $x\in \left[a,\frac{a+b}{2}\right]$, where $\lambda  =
\frac{{F\left( b \right) - F\left( a \right) - \gamma
	\left({b-a}\right)}}{{\Gamma  - \gamma }}$.
 
	for all $x \in [a,\frac{a+b}{2}]$, where $E(X)$ is the expectation
	of $X$.
\end{theorem}

\begin{proof}
	In the proof of Corollary \ref{cor3}, let $f =F$, and taking into
	account that
	\begin{align*}
		E\left( X \right) = \int_a^b {tdF\left( t \right)} = b - \int_a^b
		{F\left( t \right)dt}.
	\end{align*}
	We left the details to the interested reader.
\end{proof}

\section{Conclusion}
 	This work brings together several type of general quadrature rules; such as the general trapezoid rule or the so called Ostrowski-trapezoid, Ostrowski-Midpoint and Guessab--Schmeisser quadrature rules for symmetric points. Therefore, using the presented inequalities, several error estimates of the above mentioned quadrature rules could be deduced with appropriate numerical experiments. Hence, this work gives a very good application of the Hayashi's inequality in quadrature approximation. We left the details for the interested reader.

\centerline{}

{\bf Conflict of Interest:} The author declares that he has no conflict of interest.

\centerline{}

{\bf  Ethical approval:}  This article does not contain any studies with human participants or animals performed by any of the authors.

\centerline{}


\end{document}